\documentclass[11pt]{article}
\usepackage{amsmath, amssymb, amscd, amsthm, amsfonts}

\usepackage{graphicx}
\usepackage{hyperref}
  
\usepackage{mathtools}
\DeclareMathOperator{\M}{M}
\usepackage{siunitx}
\usepackage[title]{appendix}
\usepackage{chngcntr}
\counterwithin{figure}{section}
\counterwithin{table}{section}
\title{Algebraic Study of Discrete Imsetal Models }
\author{Amira Alkeswani}
\date{}
\newtheorem{theorem}{Theorem}[section]
\newtheorem{corollary}{Corollary}[theorem]

\newtheorem{proposition}[theorem]{Proposition}
\newtheorem{conjecture}[theorem]{Conjecture}

\theoremstyle{plain} 

\usepackage{textcomp}
\usepackage{gensymb}
\usepackage{cite}

\theoremstyle{remark}  

 \newcommand{\R}{\mathbb{R}}

\newcommand{\A}{\mathcal{A}}

\newcommand{\Z}{\mathbb{Z}}
\newcommand{\Q}{\mathbb{Q}}
\newcommand{\indep}{\rotatebox[origin=c]{90}{$\models$}}

\DeclareMathAlphabet\mathcal{OMS}{cmsy}{b}{n}

\usepackage{xfrac}
\usepackage{tikz}
\usepackage{upgreek}
\usepackage{enumitem}
\usepackage{float}

\usetikzlibrary{matrix,backgrounds,3d}
\usepackage{tikz-3dplot}
 \usepackage{tikz,stackengine,mathtools}
        \usetikzlibrary{matrix,calc}
\usepackage{ulem}
\usepackage{gensymb}
 \usepackage{multirow}
\usepackage{float}
 
\begin{document}

\maketitle

\begin{abstract}
The method of imsets, introduced by Studen\'{y}, provides a geometric and combinatorial description of conditional independence statements. Elementary conditional independence statements over a finite set of random variables are represented as column vectors of a matrix that generates a polyhedral cone. The toric ideals associated with the imsets can be used to list the conditional independence relations and identify the underlying dependencies among the random variables. In this paper, we consider a suitable discrete probability distribution over  sets of three and four binary variables, as well as a combination of binary and ternary random variables, and study the conditional independence ideals associated with these relations. Additionally, we investigate the ideals of imsetal models induced by the faces of the elementary imset cone.
\end{abstract}

\section{Introduction}
A \textit{conditional independence (CI)} statement describes the relationships among random variables in a finite set, indexed by $[n].$ It takes the form $I\indep J|K$, where $I,J$ and $K$ are disjoint subsets of $[n]$. This statement indicates the CI relationship between the variables indexed by $I$ and $J,$ given the joint probability of the variables indexed by $K$. For a fixed $n$, a CI statement is called \textit{elementary} if both $I$ and $J$ contain exactly one element; such a statement is denoted by $[i \indep j \mid K]$. Otherwise, the statement is \textit{non\_elementary.} We denote by $\mathcal{E}_n$ the set of all elementary CI statements on $n$ variables. For a given $n$, 
$
|\mathcal{E}_n| = \sigma_n = \binom{n}{2} 2^{\,n-2}.
$

Through the method of imsets, Studen\'{y} provided a  purely geometric and combinatorial description of CI statements \cite{MR3183760}. In this approach, each elementary CI statement $i\indep j |K$ can be viewed as a vector or an $\textit{imset},$ via the following linear map \cite{MR2396350}:

\vspace{0.2cm}
{\centering
  $ \displaystyle
    \begin{aligned} 
      \A :  \Z^{\sigma_n} &\mapsto \Z^{ 2^{n}}\\
    [i\indep j |K ]& \longmapsto e_{ijK}+e_K-e_{iK}-e_{jK} 
    \end{aligned}
  $ 
\par}

\vspace{0.2cm}

The set of column vectors of $\A$  generates a polyhedral cone in $\R^{ 2^{n}}$ of dimension $2^n-n-1$\cite{kashimura2012cones}. The lattice points within the cone spanned by elementary imsets with non-negative rational coefficients are called \textit{structural imsets.} A crucial type of structural imset is one that represents a non-elementary CI statement; $\textit{and}$ is a sum of two or more elementary imsets. This specific set of structural imsets is denoted by $\mathcal{S}_n.$ Representing  $[I\indep J |K] $ in $\mathcal{S}_n$ as a vector with entries of $-1$, $0$, or $1$ in $\Z^{2^{n}}$ is analogous to the process of representing elementary CI statements as vectors using the  map $\A$.

 Let $\Q[\mathcal{E}_n]$ denote  the polynomial ring over the rational numbers with generators $[i\indep j|K].$  Sturmfels and others in \cite{kashimura2012cones} have shown that the toric ideal $I_{\A_n} \in \Q[\mathcal{E}_n],$
   acts as a powerful algebraic tool for generating and analyzing conditional independence relations, effectively capturing the combinatorial structure of imsets. The process works as follows: the ideal is generated by binomials, which can be translated into implications of conditional independence.  These implications then correspond to valid algebraic equations among imsets. We call these equations 
\textit{elementary CI relations},or equivalently,\textit{toric relations.}

Let $\R[P_n],$ be the polynomial ring generated by the joint probabilities of discrete random variables indexed by $[n].$ We write $I_{[i\indep j| K]}$ for the ideal in $\R[P_n]$ associated with the CI statement corresponding to ${[i\indep j| K]}.$ When considering a finite collection of CI statements, the corresponding ideal $J$ is defined by:

\vspace{0.2cm}
{\centering
  $ \displaystyle
    \begin{aligned} 
J= I_{[i_1\indep j_1 | K_1]} +\dots +I_{[i_m\indep j_m | K_m] }.  
    \end{aligned}
  $ 
\par}
\vspace{0.2cm}

 \noindent By performing the primary decomposition of the CI ideal and studying the related independence variety, it is possible to identify certain minimal primes in the decomposition that correspond to some CI statements. This approach provides  a useful tool for constructing CI inferences\cite{drton2008lectures,MR1925796}.

 In this paper, we study collections of three and four discrete random variables, including both purely binary variables and mixtures of binary and ternary variables. Our goal is to investigate the connection between the combinatorial and geometric structure of conditional independence (CI) statements and their underlying algebraic properties. Along the way, we present several computational results and observations that we believe merit further investigation.

In particular, we explore whether a CI relation arising from $I_{\mathcal{A}_n}$ can be expressed as a relation generated by $\mathcal{S}_n$. When this is the case, we study the primary decomposition of the ideals associated with these relations and examine whether their minimal primes recover the CI ideals corresponding to the CI statements in $\mathcal{S}_n$.
 
The structure of the paper is as follows. In Section~$2$, we classify the CI relations produced by the Markov basis and the Graver basis of $I_{\mathcal{A}}$ and relate them to CI statements in $\mathcal{S}_n$. We further classify these relations according to their combinatorial structure and show in Proposition~\ref{prop:1} that each quadratic binomial gives rise to a non-elementary CI relation. In Section~$3$, we study the CI ideals associated with the CI relations computed in Section~$2$. Proposition~\ref{prop:2} shows that the CI ideals associated with elementary CI relations are isomorphic when all random variables have the same number of states. Finally, in Section~$4$, we present the results of our computations on imsetal model ideals. For $n=3$, we include all imsetal models arising from the faces of the elementary imset cone, while for $n=4$, we provide representative examples of models generated by CI statements associated with non-elementary CI relations.

\section{Conditional Independence Relations from Toric Algebra}

In this section, we classify the elementary CI relations produced by the Markov basis and the Graver basis of $I_{\A}$, and establish connections between these relations and the CI statements in 
$\mathcal{S}_n,$, particularly those that extend into non-elementary CI relations.

We use the commands [toricMarkov] and [toricGraver] in $\textit{Macaulay2}$ to compute bases associated with $\A$ and the polynomial ring $Q[\mathcal{E}_n].$  Binomial components in a Markov basis give a minimal list of CI relations, while the Graver basis $Gr_{\A_n}$ provides a larger one. Since the latter basis consists solely of primitive binomials, each generates a $\textit{unique}$ CI relation; that is, one not contained within any other relation.

 For $n>3,$ the binomial elements within a Markov basis provide a concise list of CI relations, whereas $Gr_{\A_n}$ provides a more extensive set. Since the
latter basis consists solely of primitive binomials, each generates a  \textit{distinct} CI relation. In this context, distinct means that the relation is not contained in any other.

When $n=3$, both the Markov basis and the Graver basis are equivalent. They each consist of three quadratic binomials that belong to the same symmetry class and have the form shown in Equation \eqref{eq:1}. These binomials are associated with the semigraphoid axioms and yield the following CI relation. A representative example is shown below; the remaining two follow by symmetry:

\vspace{0.2cm}
{\centering
  $ \displaystyle
    \begin{aligned}\label{eq:1}
[1 \indep 2|3 + 1 \indep 2|\emptyset] = [1 \indep 3 |2 + 1 \indep 2|\emptyset].
    \end{aligned}
  $ 
\par}
\vspace{0.2cm}
     
In the case $n=4$, the Markov basis consists of $49$ elements, which are a mix of quadratic, cubic, and quartic binomials. Among these, $24$ are quadratics and can be grouped into two symmetry classes. Examples of each class are shown below:

\begin{table}[ht]
\centering
\begin{tabular}{cc}
Class I  & $ [1 \indep 2|\emptyset \,\cdot \, 2 \indep 4|1] - [2 \indep 4 |\emptyset \,\cdot \, 1 \indep 2|4],$  \\
Class II & $ [3 \indep 4|1 \,\cdot \, 2 \indep 3|14] - [2 \indep 3|1 \,\cdot \, 3 \indep 4|12]. $
\end{tabular}
\end{table}

It also has four cubics that belong to one symmetry class, such as
 \begin{equation*} 
[2 \indep 3|1 \,\cdot \, 3 \indep 4|2 \,\cdot \, 1 \indep 3|4] -  [3 \indep 4|1 \,\cdot \, 1\indep 3|2 \,\cdot \, 2 \indep |4].
          \end{equation*}

The remaining binomial components are quartics that are partitioned into two classes. Here are representatives of each class:

 \begin{equation*} 
[1 \indep 2|4 \,\cdot \, 2 \indep 4|3 \,\cdot \, 1\indep 3|2 +3\indep 4|1] - [1 \indep 3 |4 \,\cdot \, 1 \indep 2|3 \,\cdot \, 3 \indep 4|2 \,\cdot \, 2\indep 4|1],\,\text{and}
          \end{equation*}
    \begin{equation*} 
[2\indep 4|13 \,\cdot \, 1\indep 3|4 \,\cdot \, 1\indep 4|2 \,\cdot \, 2\indep 3|\emptyset ]- [1 \indep 4 |23 \,\cdot \, 2 \indep 3|4 \,\cdot \, 2 \indep 4|1 \,\cdot \, 1\indep 3|\emptyset].
          \end{equation*}

One should notice the combinatorial structure of CI statements in these generators. Particularly, in a binomial, the random variables of CI statements of a monomial are the permutation $\gamma \in S_n$ of the random variables in the statements of the other monomial.

The set $Gr_{{\A}_4}$ consists of $3667$ generators. Not all of these elements are homogeneous; a homogeneous binomial is one whose associated vector has coordinates summing to zero. However, $2,311$ of these binomials are homogeneous, with each variable having degree one.

The linear map $\A$ is a $2^n \times \sigma _n$ matrix. Each column vector is associated with an elementary CI statement. Basic algebraic operations on these vectors do not always yield vectors representing CI statements. However, every non-elementary CI statement in the set $\mathcal{S}_n$ is a sum of at least two elementary CI statements. One way to write the imset $s$ in  $\mathcal{S}_n$ as a sum of elementary imsets is to solve the linear system $\A \cdot x=s.$ The heavy combinatorial structure of imsets allows several representations for each statement in $\mathcal{S}_n.$ A study of the characteristics of non-elementary CI statements can be found in \cite{kashimura2011properties}.

  We use combinatorics to list elements of $\mathcal{S}_n.$ In the case $n=3,$, the set $\mathcal{S}_3$ has three saturated statements belonging to one symmetry class of the form $ij\indep k | \emptyset.$  However,  $\mathcal{S}_4$ consists of $31$ statements that can be partitioned into four distinct types. The list below gives the form and the counts of CI statements for each type.
\begin{table}[h!]
\centering
\begin{tabular}{ccc}
\textbf{Type}     & \textbf{Form}                     & {[}\textbf{Count}{]} \\ \hline
Type I   & $ij \indep kl |\emptyset$ & {[}3{]}     \\
Type II  & $ijk\indep l|\emptyset$   & {[}4{]}     \\
Type III & $ij\indep k |\emptyset$   & {[}12{]}    \\
Type IV  & $ij\indep k |l$     & {[}12{]}   
\end{tabular}
\end{table}

Since each CI statement in $\mathcal{S}_n$ is a sum of at least two elementary imsets, homogeneous binomials in $I_{{\A}_n}$ with variables of degree one are candidates for representing such imsets.  We use the map $\A$ to verify whether a given binomial can be extended to a non-elementary CI relation.

 \begin{proposition} \label{prop:1}
 Each quadratic binomial of $Gr_{\A_n}$ defines a non-elementary CI relation.
 \end{proposition} 
\begin{proof}
Each quadratic binomial is associated with a semigraphoid axiom that yields the following equation \cite{MR2396350}
$$[i\indep j|K\cup l+ i\indep l|K] =[i\indep j|K + i\indep l|K\cup j],$$ 
The image of  CI statements in the above equation  under the map $\A$ is:
\begin{footnotesize}
\begin{align*}
 [( e_{ijlK}+e_{lK}-e_{ilK}-e_{jlK}) + (e_{ilK}+e_{K}-e_{iK}-e_{lK})] &=  [(e_{ijK}+e_K-e_{iK}-e_{jK}) + (e_{ijlK}+e_{jK}-e_{ijK}-e_{jlK})]\\
     [e_{ijlK} + e_{K} - e_{iK} -  e_{jlK}]& = [e_{ijlK} + e_{K} - e_{iK} -  e_{jlK}]\\
     [i\indep jl|K] &= [i\indep jl|K].\\
 \end{align*}
  \end{footnotesize}
 Thus
$$[i\indep j|K\cup l+ i\indep l|K] =[i\indep j|K + i\indep l|K\cup j]= [ i\indep jl |K].$$ 
\end{proof}
Using the software \textit{Macaulay2} and the above proposition, we derive the following result.
 \begin{corollary} \label{corr:1}
 For $n=3,$ the set $\mathcal{E}_3$ is in bijection with the set of semigraphoid axioms.
 \end{corollary}
 
\begin{proof}
  The computation showed that the Markov basis of $ I_{\mathcal{A}_3} $ is identical to
  $Gr_{\mathcal{A}_3};$ both consist of three quadratic binomials of the form given in Equation \eqref{eq:1}. The set $ \mathcal{S}_3$ contains three statements of the form $i \indep j k | \emptyset$.
  
 Consider the map $ \Phi $ for $ n=3 $, defined by
  \begin{align*}
    \Phi : & Gr_{\A_n} \to \mathcal{S}_n.
  \end{align*}
  The map $\Phi $ is injective by Proposition~\ref{prop:1}. To show that \( \Phi^{-1} \) is injective, let $ \textbf{s}_i $ be the vector that represents a CI statement in $ \mathcal{S}_3 $. For each $i \in [3]$, we solved the system $ \A \cdot \textbf{x}_i = \textbf{s}_i $. Each resulting $ \textbf{x}_i $ corresponded  to a unique initial monomial in $ Gr_{A_3} $.
\end{proof}

The above discussion implies that the only non-elementary relation for $n=3$ takes the form
 \begin{equation} \label{eq:51}
 [i \indep j|k + i \indep j|\emptyset]= [i \indep k |j+ i \indep j|\emptyset]=[i \indep jk|\emptyset].
          \end{equation}
          
Corollary $\ref{corr:1}$ does not hold for $n=4.$ The map $\Phi$ is injective since $|\mathcal{S}_n|>24.$ The CI relations produced by the Class I and Class II quadratic binomials  in $Gr_{\A_4}$ yield statements of Type III and Type IV, respectively, in $\mathcal{S}_n$. Two representative examples are given below:
 
 \begin{equation} \label{eq:6}
 [1 \indep 2|\emptyset + 2 \indep 4|1] = [2 \indep 4 |\emptyset+ 1 \indep 2|4] = [1 4\indep 2|\emptyset].
          \end{equation}
          \begin{equation} \label{eq:7}
[1 \indep 2|3 + 2 \indep 4|13] = [2 \indep 4 |3+ 1 \indep 2|34] = [1 4\indep 2|3].
          \end{equation}
We observed that the type of CI statement in \(\mathcal{S}_n\) associated with each quadratic binomial 
is determined by the permutation \(\gamma \in S_n\) that rearranges the random variables of 
the CI statements within its monomial. It also depends on the variables being conditioned on. 
For instance, in Equation~(3), the random variables indexed by \(1\) and \(4\) in the statement 
\(14 \perp 2 \mid 3\) arise because the transposition \((1, 4)\) is applied to the variables in each 
statement on one side of the equation; that is, within one monomial, to produce those on 
the other side. Moreover, the statement \(14 \perp 2 \mid 3\) conditions only on \(3\), since it is 
the index shared across all the elementary CI statements in that relation. This observation 
provides a more efficient way to construct non-elementary relations from the quadratic 
components of \(Gr_{A_n}\) than by directly using the linear map  \(\A\).

The cubic and quartic binomials produced by the Markov basis of $I_{\A_4}$ do not yield non-elementary CI relations.
We attempted to solve for the saturated CI statements in $\mathcal{S}_4$ of Type I and Type II using \textit{Macaulay2}, but obtained vectors with supports other than $\mp 1.$ Therefore, we used the linear map $\A$ to manually compute representations of each of these statements as sums of elementary CI statements. The Appendix provides a list of non-elementary CI relations corresponding to each type. Two examples are shown below:
\begin{equation}  \label{eq:2}
 [2 \indep 4|13 + 1 \indep 4|3\emptyset+3 \indep 4|\emptyset] = [3 \indep 4|12 + 1 \indep 4|2\emptyset+2 \indep 4|\emptyset] =[123\indep 4|\emptyset],\,\text{and}
\end{equation}
 \begin{equation} \label{eq:3}
 [1 \indep 3|24 + 1 \indep 2|4\emptyset+2 \indep 4|3+2\indep3|\emptyset] = [1 \indep 4|23 + 1 \indep 2|3\emptyset+2 \indep 3|4+2\indep4|\emptyset] = [12\indep 34|\emptyset].
\end{equation}

For $n=5,$ the Appendix of \cite{MR2396350} includes the $120$ quadratic binomials produced by a Markov basis of $I_{\A_5}$. According to Proposition \ref{prop:1}  each of these binomials defines  a CI statement in $\mathcal{S}_5.$ The following is an example: 
$$ [3\indep 5|12 + 3\indep 4|125] =[ 3\indep 4|12 + 3\indep 5|124] = [45\indep 3|12].$$ 

In the following section, we study the CI ideals in the polynomial ring $\R[P_n]$ associated with the CI relations induced by $I_{\A}.$ 

\section{ The CI Ideals of Toric Relations } \label{s:3}
The core of this study is to connect the geometric and combinatorial structure of the CI statements described by the method of imsets with their algebraic aspects. Studen\'{y}'s framework suggests that the geometric description of CI statements can guide the computation of the primary decomposition of CI ideals, even in cases where computer algebra systems may not provide complete results. To interpret a collection of CI statements modeled by imsets as probabilistic CI statements, we assume that the probability distribution over $n$ belongs to the class of distributions with finite multiinformation, as discussed in \cite{MR3183760}.

Let $J_i \subset \mathbb{R}[P_n]$ be the ideal defined as the sum of the CI ideals corresponding to the CI statements that verify one side of the equation, as shown in Equations~\ref{eq:2} and~\ref{eq:3}. In this section, we characterize the ideals $J_1, J_2,$ and $J_3$ in $\mathbb{R}[P_n]$, defined as follows:
$$ \underbrace{I_{ [i_{s_1}\indep j_{s_1} | K_{s_1}]} +\dots +I_{[i_{s_q}\indep j_{s_q} | K_{s_q}]}}_{J_1} ,\,\,\underbrace{I_{[i_{t_1}\indep j_{t_1} | K_{t_1}]} + \dots + I_{[i_{t_r}\indep j_{t_r} | K_{t_r}]}}_
{J_2}\,\,\text{and}\,\, \underbrace{I_s}_{J_3}.$$
 
 We used the library [primdec.lib] in $\textit{Singular}$ and the commands [primdecSY], [minAssGTZ], 
[dim], [Equal(ideal I, ideal J)], and [degree] to compute the primary decomposition, minimal primes, degree, and dimension of these ideals.

 $$  \underbrace{I_{1 \indep 2|3} + I_{1 \indep 2|\emptyset}}_{J_1},\,\, \underbrace{I_{ 1 \indep 3 |2}+ I_{1\indep 2|\emptyset}}_{J_2},\,\,\text{and}\,\,\underbrace{I_{1 \indep 23|\emptyset}}_{J_3}.$$

  For these computations, we first considered all random variables to be binary, and then examined mixed cases where one variable was ternary. The summarized results are shown in Table \ref{tab:1} below.
  
\begin{table}[ht]
\centering
\begin{tabular}{c|ccc|ccc|ccc|ccc|}
\cline{2-13}
                                     & \multicolumn{3}{c|}{\textbf{(2,2,2)}}                           & \multicolumn{3}{c|}{\textbf{(3,2,2)}}                           & \multicolumn{3}{c|}{\textbf{(2,3,2)}}                           & \multicolumn{3}{c|}{\textbf{(2,2,3)}}                           \\ \cline{2-13} 
                                     & \multicolumn{1}{c|}{$J_1$} & \multicolumn{1}{c|}{$J_2$} & $J_3$ & \multicolumn{1}{c|}{$J_1$} & \multicolumn{1}{c|}{$J_2$} & $J_3$ & \multicolumn{1}{c|}{$J_1$} & \multicolumn{1}{c|}{$J_2$} & $J_3$ & \multicolumn{1}{c|}{$J_1$} & \multicolumn{1}{c|}{$J_2$} & $J_3$ \\ \hline
\multicolumn{1}{|c|}{\textbf{dim}}            & \multicolumn{1}{c|}{9}  & \multicolumn{1}{c|}{9}  & 5  & \multicolumn{1}{c|}{7}  & \multicolumn{1}{c|}{7}  & 6  & \multicolumn{1}{c|}{7}  & \multicolumn{1}{c|}{8}  & 7  & \multicolumn{1}{c|}{8}  & \multicolumn{1}{c|}{7}  & 7  \\ \hline
\multicolumn{1}{|c|}{\textbf{degree}}         & \multicolumn{1}{c|}{2}  & \multicolumn{1}{c|}{2}  & 4  & \multicolumn{1}{c|}{6}  & \multicolumn{1}{c|}{6}  & 10 & \multicolumn{1}{c|}{6}  & \multicolumn{1}{c|}{2}  & 6  & \multicolumn{1}{c|}{2}  & \multicolumn{1}{c|}{6}  & 6  \\ \hline
\multicolumn{1}{|c|}{\#\textbf{minimal primes}} & \multicolumn{1}{c|}{3}  & \multicolumn{1}{c|}{3}  & 1  & \multicolumn{1}{c|}{3}  & \multicolumn{1}{c|}{3}  & 1  & \multicolumn{1}{c|}{3}  & \multicolumn{1}{c|}{7}  & 1  & \multicolumn{1}{c|}{7}  & \multicolumn{1}{c|}{3}  & 1  \\ \hline
\end{tabular}
\caption{\label{tab:1} Summary of the computation on the CI ideals over $n=3.$}
\end{table}

The computation produces the following observations:
\begin{enumerate}[label=(\roman*)]
    \item The ideals $J_1,J_2$ and $J_3$ are not equal.
    \item Examining the sets of minimal primes of \(J_1\) and \(J_2\) shows that their intersection 
contains a binomial ideal precisely the prime ideal \(J_3\). 
The ideal \(J_3\) corresponds to the \(2 \times 2\) minors of the associated probability 
matrix of size \(r_1 \times r_2\). 
The variety \(V(J_3)\) is a projective toric variety given by the Segre embedding of 
\(\mathbb{P}^{r_1 - 1} \times \mathbb{P}^{r_2 - 1}\) into \(\mathbb{P}^8\). 
In fact, this variety is the only subvariety of \(V(J_1)\) and \(V(J_2)\) that intersects 
the probability simplexes \(\Delta_7\) and \(\Delta_{11}\) non-trivially. 
The vanishing sets of the remaining components lie along the boundaries of these simplexes. 
Therefore, the following CI inferences hold:

 $$ [i \indep j|k \,\,\text{and}\,\, i \indep j|\emptyset] \implies [i \indep jk |\emptyset],\,\text{and} $$
    $$ [i\indep k |j \,\,\text{and}\,\, i \indep j|\emptyset] \implies  [i \indep jk |\emptyset],$$
\end{enumerate}
are verified by the primary decomposition of $J_1$ and $J_2.$
The chances for \textit{Singular} to compute the primary decomposition of a given CI ideal decrease with the increase of $n$ and the values of random variables. Although in the case $n=3,$ the commands \texttt{primdecSY}, \texttt{minAssGTZ}, did not yield results when all the variables were ternary. 
 
When \(n = 4,\) we present results only for the CI ideals corresponding to the 
$24$ quadratic binomials and for selected CI relations from the Markov basis. 
 
 Starting with the quadratic CI relations of the semigraphoid axioms, we tested the CI ideals associated with \ref{eq:2} and \ref{eq:3} when all random variables are binary and when the second variable is ternary.
 
$$  \underbrace{I_{1 \indep 2|\emptyset} + I_{2 \indep 3|1}}_{P_1}, \,\, \underbrace{ I_{ 2 \indep 4 |\emptyset}+ I_{1 \indep 3|2}}_{P_2}, \,\, \text{and} \, \underbrace{I_{1 4\indep 2|\emptyset}}_{P_3}$$
$$  \underbrace{I_{3 \indep 4|1} + I_{2 \indep 3|14}}_{Q_1},\,\,  \underbrace{ I_{2 \indep 3 |1}+ I_{2 \indep 4|13}}_{Q_2},\,\,\text{and}\,\,\underbrace{ I_{2 4\indep 3|1}}_{Q_3}$$

The results of computing the degree, dimension, and total number of minimal primes 
of the ideals \(P_i\) and \(Q_i\) are summarized in Table~\ref{tab:2}.

\begin{table}[h!] 
\centering
\begin{tabular}{c|cccccc|cccccc|}
\cline{2-13}
                                                 & \multicolumn{6}{c|}{\textbf{(2,2,2,2)}}                                                                                                                & \multicolumn{6}{c|}{\textbf{(2,3,2,2)}}                                                                                                                \\ \cline{2-13} 
                                                 & \multicolumn{1}{c|}{$P_1$} & \multicolumn{1}{c|}{$P_2$} & \multicolumn{1}{c|}{$P_3$} & \multicolumn{1}{c|}{$Q_1$} & \multicolumn{1}{c|}{$Q_2$} & $Q_3$ & \multicolumn{1}{c|}{$P_1$} & \multicolumn{1}{c|}{$P_2$} & \multicolumn{1}{c|}{$P_3$} & \multicolumn{1}{c|}{$Q_1$} & \multicolumn{1}{c|}{$Q_2$} & $Q_3$ \\ \hline
\multicolumn{1}{|c|}{\textbf{dim}}               & \multicolumn{1}{c|}{13}    & \multicolumn{1}{c|}{13}    & \multicolumn{1}{c|}{13}    & \multicolumn{1}{c|}{10}    & \multicolumn{1}{c|}{10}    & 10    & \multicolumn{1}{c|}{19}    & \multicolumn{1}{c|}{19}    & \multicolumn{1}{c|}{18}    & \multicolumn{1}{c|}{14}    & \multicolumn{1}{c|}{17}    & 12    \\ \hline
\multicolumn{1}{|c|}{\textbf{degree}}            & \multicolumn{1}{c|}{2}     & \multicolumn{1}{c|}{2}     & \multicolumn{1}{c|}{4}     & \multicolumn{1}{c|}{4}     & \multicolumn{1}{c|}{4}     & 4     & \multicolumn{1}{c|}{3}     & \multicolumn{1}{c|}{3}     & \multicolumn{1}{c|}{10}    & \multicolumn{1}{c|}{36}    & \multicolumn{1}{c|}{4}     & 100   \\ \hline
\multicolumn{1}{|c|}{\textbf{\# minimal primes}} & \multicolumn{1}{c|}{3}     & \multicolumn{1}{c|}{3}     & \multicolumn{1}{c|}{1}     & \multicolumn{1}{c|}{9}     & \multicolumn{1}{c|}{9}     & 1     & \multicolumn{1}{c|}{9}     & \multicolumn{1}{c|}{9}     & \multicolumn{1}{c|}{1}     & \multicolumn{1}{c|}{9}     & \multicolumn{1}{c|}{49}    & 1     \\ \hline
\end{tabular}

\caption{\label{tab:2}Summary of the computation on the CI ideals over $n=4.$}
\end{table}

The following statements summarize the most significant observations regarding the relationships among these CI ideals. 

\begin{enumerate}
    \item The ideals \(P_1, P_2,\) and \(P_3\) are distinct, as are the ideals \(Q_1, Q_2,\) and \(Q_3.\)

    \item When all variables are binary, each pair of ideals \((P_1, P_2)\) and \((Q_1, Q_2)\) 
    share the same degree and dimension, and decompose into the same number of minimal primes.

    \item The minimal primes of \(P_1\) and \(P_2\) include one component that 
    intersects the probability simplex. This component corresponds precisely to the ideal \(P_3.\) 
    A similar relationship holds among \(Q_1, Q_2,\) and \(Q_3.\)

    \item Attempts to compute the CI ideals associated with the cubic and quartic relations 
    provided only their dimension and degree, as the software could not complete 
    the full primary decomposition.
\end{enumerate}
 
For the remaining CI relations that each yield a CI statement in \(E_4\), 
as shown in Equations~(4) and~(5), the computations confirmed that the 
ideals \(I_{123 \perp 4 \mid \emptyset}\) and \(I_{12 \perp 34 \mid \emptyset}\) 
lie in the intersection of the minimal primes of the CI ideals 
\(P_1\) and \(P_2,\) and of the ideals \(Q_1\) and \(Q_2,\) respectively. 
In particular, we have

\[
I_{123 \perp 4 \mid \emptyset} \subset 
\underbrace{ I_{2 \perp 4 \mid 13} + I_{1 \perp 4 \mid 3\emptyset} + I_{3 \perp 4 \mid \emptyset} }_{P_1}
\cap
\underbrace{ I_{3 \perp 4 \mid 12} + I_{1 \perp 4 \mid 2\emptyset} + I_{2 \perp 4 \mid \emptyset} }_{P_2},
\]
\[
I_{12 \perp 34 \mid \emptyset} \subset
\underbrace{ I_{1 \perp 3 \mid 24} + I_{1 \perp 2 \mid 4\emptyset} + I_{2 \perp 4 \mid 3} + I_{2 \perp 3 \mid \emptyset} }_{Q_1}
\cap
\underbrace{ I_{1 \perp 4 \mid 23} + I_{1 \perp 2 \mid 3\emptyset} + I_{2 \perp 3 \mid 4} + I_{2 \perp 4 \mid \emptyset} }_{Q_2}.
\]

\noindent The above discussion leads to the following generalization.

\begin{proposition} \label{prop:2}
Consider an elementary relation and assume that all random variables indexed by \(n\) 
take the same number of possible values. 
If \(l = t\) and, for every \(a \in [l]\), the CI statement \(C_a\) can be obtained by permuting 
the random variables of \(D_b\) for some \(b \in [t],\) 
then the CI ideals associated with the statements on both sides of the elementary relation 
are isomorphic.
\end{proposition}

\begin{proof}
Suppose all random variables take the same value \(r.\)
The probability matrices corresponding to each CI statement in the relation 
are square matrices of size \(r \times r.\)
Let \(J_1, J_2 \subset R[P_n]\) be defined as
\[
J_1 = I_{[i_{l_1} \perp j_{l_1} \mid K_{l_1}]} + \cdots + I_{[i_{l_q} \perp j_{l_q} \mid K_{l_q}]},
\qquad
J_2 = I_{[i_{t_1} \perp j_{t_1} \mid K_{t_1}]} + \cdots + I_{[i_{t_r} \perp j_{t_r} \mid K_{t_r}]}.
\]
Since the number and types of CI statements are the same, the ideals \(J_1\) and \(J_2\) 
have the same number of generators. 
If the indeterminates of the generators of \(J_1\) are of the form \(p_{r_1 \dots r_n}\) 
for \(r_i \in [r],\) then those of \(J_2\) are \(p_{\gamma(r_1 \dots r_n)}\) 
for some permutation \(\gamma \in S_n.\) 
Therefore, the two ideals are isomorphic.
\end{proof}

The study in~\cite{MR2825749} showed that for elementary relations that can be extended 
to non-elementary CI relations, the collections of elementary CI statements 
on both sides of the equation are related by a permutation of the random variables. 
We observed that only the homogeneous binomials in \(Gr_{A_4}\), where all variables 
have degree one, produce CI relations of this kind. 
Using the software \textit{Singular}, we derived the following result.

\begin{proposition} \label{prop:3}
For a set of binary and mixed binary and ternary random variables indexed by \(n = 3, 4,\) 
let \(J_1, J_2,\) and \(J_3 \subset R[P_n]\) be the CI ideals that satisfy a three-sided CI relation such that
\[
J_1 = I_{[i_{l_1} \perp j_{l_1} \mid K_{l_1}]} + \cdots + I_{[i_{l_q} \perp j_{s_q} \mid K_{s_q}]}, \,
J_2 = I_{[i_{t_1} \perp j_{t_1} \mid K_{t_1}]} + \cdots + I_{[i_{t_r} \perp j_{t_r} \mid K_{t_r}]}, \text{and} \quad
J_3 = I_s.
\]
The CI ideal \(J_3\) lies in the intersection of the minimal primes of \(J_1\) and \(J_2.\) 
Moreover, the variety \(V(J_3)\) is the only set that intersects the probability simplex non-trivially.
\end{proposition}

The following example illustrates a CI relation produced by \(Gr_{A_4}\) 
that does not correspond to any CI statement in \(E_4:\)
\[
I_{2 \perp 4 \mid 13} + I_{1 \perp 4 \mid 3} + I_{4 \perp 3 \mid \emptyset} 
= P_1 
= I_{4 \perp 3 \mid 12} + I_{1 \perp 4 \mid 2} + I_{2 \perp 4 \mid \emptyset} 
= P_2.
\]
The computation shows that \(P_1\) and \(P_2\) are isomorphic in the binary case, 
but have different degrees, dimensions, and minimal primes otherwise. 
The converse of Proposition \ref{prop:3} is not true, as demonstrated by a counterexample 
in Table~\ref{tab:2}.

\section{ Ideals of Imsetal models}
Every face of the elementary imset cone induces a unique imsetal model. 
In the case \(n = 3,\) every imsetal model can be interpreted as a probabilistic model. 
For \(n = 4,\) however, Studený showed that there exists a structural imset \(u\) such that 
\(M_u \neq M_P\) for any probability distribution \(P.\) 
By computing the face lattice of the cone of elementary imsets, one can list the 
elementary CI statements in each model, since they correspond to the generators 
of each face. 
The set of faces of dimension \(d\) produces a collection of imsetal models denoted by \(M_d,\) 
where each model \(M_d^{\alpha} \in M_d\) is generated by at least \(d\) elementary CI statements.

Assuming that the collection \(M_d\) consists of probabilistic models, 
we associate the CI statements in each \(M_d^{\alpha}\) with the CI relations 
induced by \(I_A.\) 
We also examine the CI ideals in \(R[P_n]\) corresponding to every imsetal model 
for \(n = 3,\) and analyze one representative example for \(n = 4.\)

We used the commands \texttt{F VECTOR} and \texttt{DIAGRAM $\rightarrow$ FACES}
in \textit{Polymake} to compute the face lattice of the cone for \(n = 3\) and \(n = 4.\) 
By examining the collection of elementary CI statements that generate each face, 
we observed that the set of CI statements on one side of any CI relation 
does not generate a face of the cone. 
Furthermore, the four CI statements associated with each quadratic binomial 
(corresponding to the semigraphoid axiom) generate a three-dimensional face.

For \(n = 3,\) the \(f\)-vector of the elementary imset cone is \((0, 6, 9, 5, 1).\) 
The faces of the cone form a partially ordered set (poset), as illustrated in Figure~\ref{fig:4}. 
This cone lives in \(\mathbb{R}^8\) and has dimension four. 
We label the models belonging to the set of faces \(F_d\) of dimension \(d\) by \(M_d.\) 
For example, as shown in Figure~\ref{fig:1}, the set \(F_1\) produces six models, 
each consisting of a single elementary CI statement. 
The subscript \(\alpha\) in \(M_d^{\alpha}\) indicates the position of the model \(\alpha\) 
in the Hasse diagram, counted from left to right. 
For instance, \(M_2^{4} \subset M_2\) represents the model that includes 
the statements \(1 \perp 2 \mid 3\) and \(1 \perp 2 \mid \emptyset,\) 
while \(M_3^{5}\) consists of three saturated statements. 
We denote by \(M_3^{*} = \{M_3^{2}, M_3^{3}, M_3^{4}\}\) 
the subset of models in \(M_3\) associated with the semigraphoid axioms, 
which are marked by double-bordered nodes in the diagram.

  \begin{figure}[H]
   \label{fig:1}
    \centering
    \label{fig:4}
\includegraphics[width=.8\textwidth,height=.5\textheight]{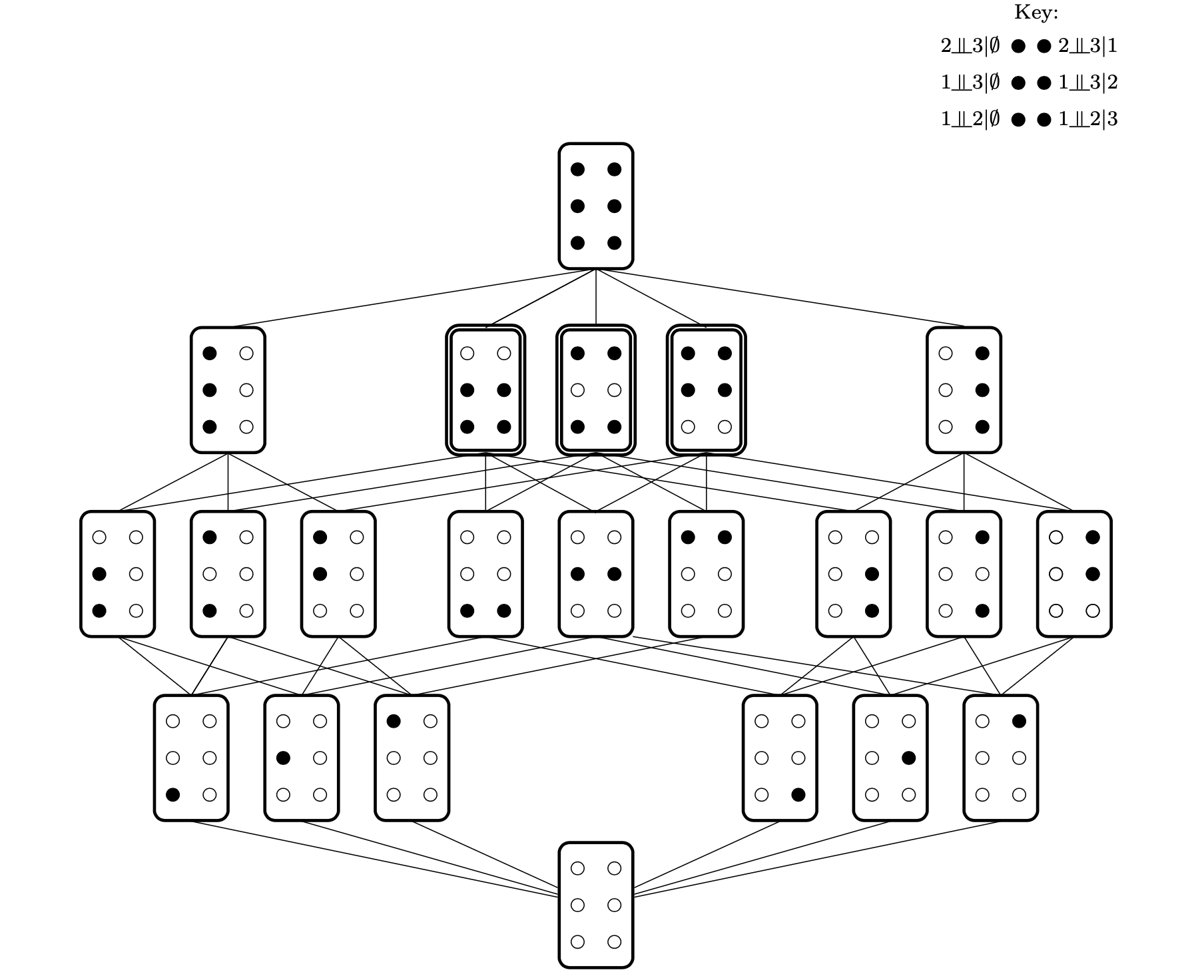}
\caption{\label{figure:figure-name1}The lattice of CI models for $n=3.$}
 \end{figure}
To reduce the computational effort required for analyzing the CI ideals \(I_{M_d},\) 
we divided the set of models \(M_d\) into equivalence classes according to the types 
of elementary CI statements contained in each model. 
The set \(M_1\) consists of two classes: one containing three marginal statements 
and the other containing three saturated statements. 
The sets \(M_2\) and \(M_3\) are each divided into three equivalence classes. 
Since CI ideals belonging to the same equivalence class are isomorphic 
by Proposition \ref{prop:2}, we selected a representative model \(M_d^{\alpha}\) 
from each class to examine the properties of its ideal \(I_{M_d^{\alpha}}.\) 
Table \ref{tab:3} presents the equivalence classes and summarizes 
the results of our computations when all random variables are binary. 
We also tested \(I_{M_d^{\alpha}}\) in mixed cases where two variables are binary 
and the third is ternary.
\vspace{.7cm}

\begin{table}[ht]
\centering
\begin{tabular}{|c|c|c|c|c|c|}
\hline
\textbf{CI Ideal} & \textbf{dim.} & \textbf{degree} & \textbf{Is Prime} & \textbf{\# Min. Primes} & \textbf{Min. Primes dim.} \\ \hline \hline
\(I_{M_{1}^{1}}\) & 7 & 2 & \small{Yes} & 1 & 7 \\ \hline
\(I_{M_{1}^{6}}\) & 6 & 4 & \small{No} & 1 & 6 \\ \hline \hline
\(I_{M_{2}^{1}}\) & 6 & 4 & \small{No} & 2 & 6, 6 \\ \hline
\(I_{M_{2}^{4}}\) & 5 & 8 & \small{No} & 2 & 5, 5 \\ \hline
\(I_{M_{2}^{9}}\) & 5 & 4 & \small{No} & 3 & 4, 4, 5 \\ \hline \hline
\(I_{M_{3}^{1}}\) & 5 & 8 & \small{No} & 4 & 5, 5, 5, 5 \\ \hline
\(I_{M_{3}^{2}}\) & 5 & 4 & \small{Yes} & 1 & 5 \\ \hline
\(I_{M_{3}^{5}}\) & 4 & 5 & \small{Yes} & 4 & 2, 2, 2, 4 \\ \hline \hline
\(I_{M_{4}}\) & 4 & 6 & \small{Yes} & 1 & 4 \\ \hline
\end{tabular}
\caption{\label{tab:3} Summary of the computation on the CI ideals of the binary models \(M_{d}^{\alpha}\) over \(n = 3.\)}
\end{table}
The following statements on imsetal models over $n=3$ are proved by computation.

\begin{enumerate}[label=(\roman*)]
\item For any model \(M_d^{\alpha}\) with \(d > 1\) that includes at least one saturated 
CI statement, the primary decomposition of its ideal \(I_{M_d^{\alpha}}\) 
contains at least one binomial component. 
These binomial components are equal to \(I_s\) for some \(s \in \mathcal{S}_n.\) 
The form and number of such components depend on how these models 
relate to the three models \(M_3^{*}\) that are generated by the four elementary 
CI statements described in Equation~\ref{eq:2} within the lattice structure. 
The following list presents the most representative examples of these models.

\begin{itemize}

\item We start with a representative of the three models in \(M_3^{*}.\)
\[
M_3^{2} = \{1 \perp 2 \mid \emptyset,\; 2 \perp 3 \mid 1,\; 2 \perp 3 \mid \emptyset,\; 1 \perp 2 \mid 3\}.
\]
The elements of this set are involved in the following non-elementary CI relation:
\[
1 \perp 2 \mid 3 + 1 \perp 3 \mid \emptyset = 1 \perp 3 \mid 2 + 1 \perp 2 \mid \emptyset = 12 \perp 3 \mid \emptyset.
\]
The computation shows that the primary decomposition of \(I_{M_3^{2}}\) includes the following binomial ideal:
\begin{align*}
I_{12 \perp 3 \mid \emptyset} 
&= \langle p_{212}p_{221} - p_{211}p_{222},\; 
p_{122}p_{212} - p_{112}p_{222},\;
p_{121}p_{212} - p_{111}p_{222}, \\
&\quad p_{122}p_{211} - p_{112}p_{221},\;
p_{121}p_{211} - p_{111}p_{221},\;
p_{112}p_{121} - p_{111}p_{122} \rangle.
\end{align*}
In the binary case, \(I_{M_3^{2}} = I_{12 \perp 3 \mid \emptyset}.\)
We used the command \texttt{gens gb I} in \textit{Macaulay2} to compute a Gr\"obner basis for \(I_{M_3^{2}}.\)
The computation confirms that the generators of \(I_{12 \perp 3 \mid \emptyset}\) form a Gr\"obner basis for \(I_{M_3^{2}}\) 
with respect to the lexicographic term order.
For other configurations of random variable values, the ideal \(I_{M_3^{2}}\) is not prime; 
the variety of each non-binomial minimal prime intersects the boundary of the probability simplex \(\Delta_{11}.\)

\item In the binary case, the full model 
\(M_4 = M_3^{2} \cup M_3^{3} \cup M_3^{4}\) 
is a toric ideal~\cite{drton2008lectures}. 
This ideal is exactly the binomial ideal
\[
I_{\mathfrak{E}} = I_{13 \perp 2 \mid \emptyset} + I_{12 \perp 3 \mid \emptyset} + I_{23 \perp 1 \mid \emptyset}.
\]
Moreover, the binomials of \(I_{\mathfrak{E}}\) generate a Gr\"obner basis for \(I_{M_4}\) 
with respect to the lexicographic term order.
When one of the random variables is ternary, the ideal \(I_{\mathfrak{E}}\) 
remains the only component in the primary decomposition of \(I_{M_4}\) 
whose variety satisfies \(\mathcal{V}(I_{\mathfrak{E}}) \subset \Delta_{11} \setminus \{0\}.\)

\item An interesting observation is that \(I_{M_3^{5}} \subset I_{M_4}\) 
for the model \(M_3^{5}\), which consists of the three elementary saturated CI statements 
that contain \(I_{\mathfrak{E}}\) in their primary decomposition.

\item Regarding the models \(M_2^{\alpha}\) for \(\alpha > 3:\)
The type of binomial ideal(s) in the primary decomposition of \(I_{M_2^{\alpha}}\) 
depends on the model \(M_3^{\beta} \in M_3^{*}\) such that \(M_2^{\alpha} \subset M_3^{\beta}.\)
For example, consider the following models:
\begin{align*}
M_2^{4} &= \{1 \perp 2 \mid 3,\; 1 \perp 2 \mid \emptyset\}, \\
M_2^{6} &= \{1 \perp 2 \mid 3,\; 1 \perp 3 \mid 2\}.
\end{align*}

The model \(M_2^{4}\) is a submodel of both \(M_3^{2}\) and \(M_3^{3}.\)
The computation shows that the binomial ideals \(I_{12 \perp 3 \mid \emptyset}\) 
and \(I_{23 \perp 1 \mid \emptyset}\) appear in the primary decomposition of \(I_{M_2^{4}}.\)
Since the model \(M_2^{6}\) is a submodel of \(M_3^{2}\) only, 
and following Sturmfels~\cite{eisenbud1996binomial}, 
we expected \(I_{M_2^{6}}\) to decompose into more binomial components, 
including \(I_{12 \perp 3 \mid \emptyset}.\)
However, \(I_{12 \perp 3 \mid \emptyset}\) was the only binomial component 
that appeared in the primary decomposition.
This observation supports the idea that the algebraic and combinatorial 
representation of CI statements through imsets 
can effectively identify the minimal primes of the associated CI ideals.
For CI inferences of these models, see \S6.7 in~\cite{drton2008lectures}.
\end{itemize}

\item The ideals of the models that consist only of marginal statements, 
such as \(I_{M_2^{1}}, I_{M_2^{2}}, I_{M_2^{3}},\) and \(I_{M_3^{1}},\) 
are complete intersections. 
Theorem $4.3.5$ in \cite{sullivant2018algebraic} provides a detailed description 
of the varieties associated with these ideals.

\end{enumerate}

In the case $n=4,$ the total number of faces is $22108.$ The following is the $f$-vector of the elementary imset cone associated with $\A_4:$
  $$(1, 24, 228, 1128, 3212, 5560, 5980, 3985, 1596, 356, 37, 1).$$
   It was shown that not every imsetal model can be considered a probabilistic model. In particular, we have more initial models than probabilistic models. Studen\'{y} called these models \textit{facial models}. He showed by the counterexample, Example $4.1$ in \cite{MR3183760}, that there exists a structural imset $u$ such that $\M_u\neq \M_P$  for any probability distribution $P.$
We computed the CI ideals of the models associated with semigraphoids found by the Markov basis 
in Section~\ref{s:3} and their submodels, considering two combinations of values for the random variables: 
all binary, and when the second variable is ternary. 
In the binary case, due to the ideal isomorphism proved in Proposition~\ref{prop:2}, 
we selected the following two models to represent Class~I:
\[
M_3^{1} = \{ 2 \perp 4 \mid 1,\; 2 \perp 3 \mid 14,\; 2 \perp 3 \mid 1,\; 2 \perp 4 \mid 13 \},
\]
\[
M_3^{2} = \{ 2 \perp 4 \mid 1,\; 3 \perp 4 \mid 12,\; 3 \perp 4 \mid 1,\; 2 \perp 4 \mid 13 \}.
\]

Note that \(M_3^{1}\) is generated by the elementary CI statements appearing 
in the three-sided relations of Equation~\ref{eq:6}. 
Figure~\ref{fig:1} illustrates the nodes of these models and their submodels.
\begin{figure}[ht]
    \centering
    \includegraphics[width=.6\textwidth,height=.3\textheight]{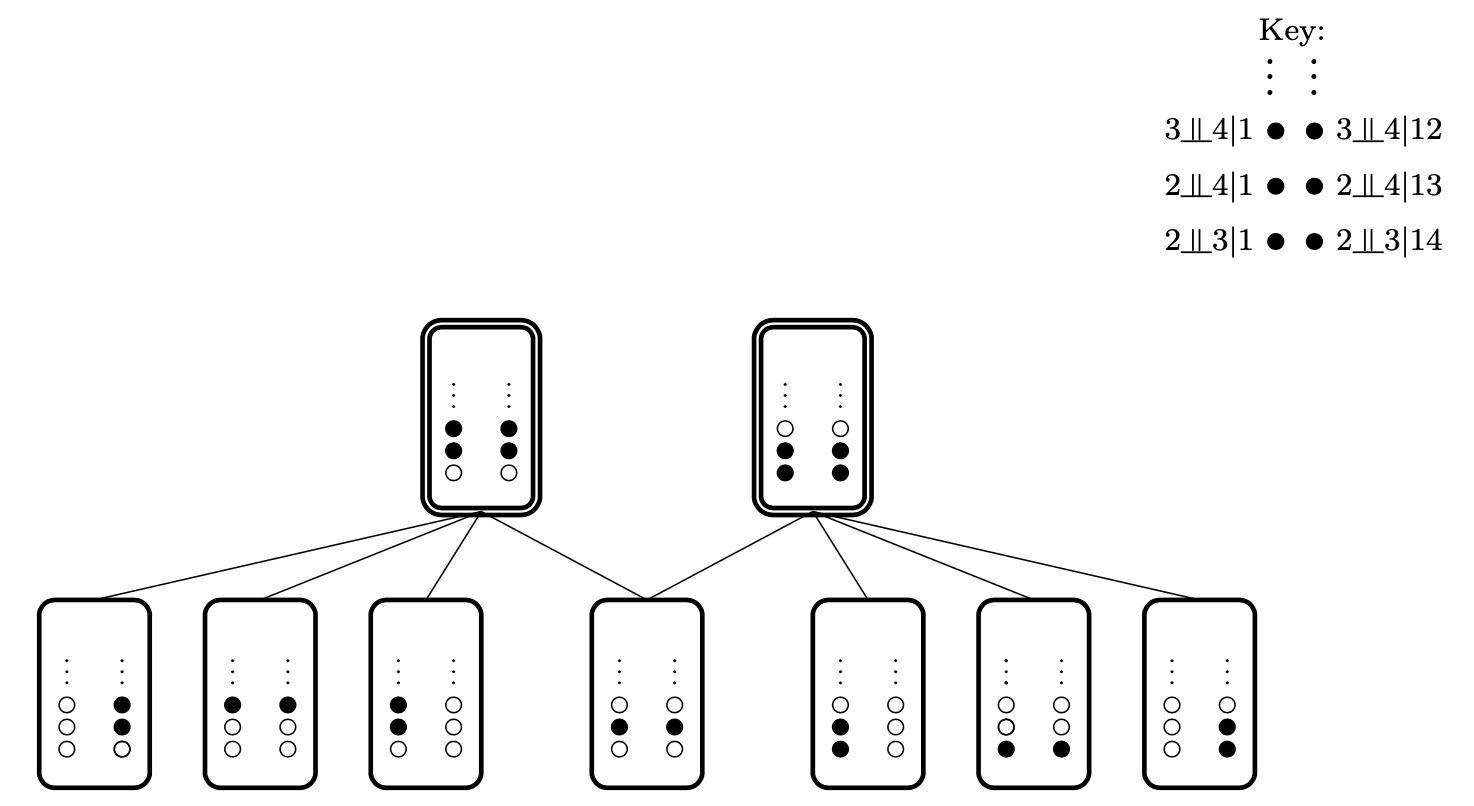}
    \caption{\label{figure:figure-name2} Sub-lattice of CI models for \(n = 4.\)}
\end{figure}

We obtained results similar to those for \(n = 3.\) 
A summary of the most significant findings is provided below.

\begin{enumerate}[label=(\roman*)]

\item In the binary case, the following equalities are verified by computation:
\[
I_{M_3^{1}} = I_{34 \perp 2 \mid 1},
\qquad
I_{M_3^{2}} = I_{32 \perp 4 \mid 1}.
\]

\item The prime decomposition of every submodel \(I_{M_2^{\alpha}}\) 
that includes at least one saturated CI statement contains binomial components, 
specifically \(I_{32 \perp 4 \mid 1}\) and \(I_{34 \perp 2 \mid 1}.\) 
This decomposition depends on the poset structure, 
particularly on the connection of each submodel to \(M_3^{2}\) and \(M_3^{3}.\) 
For example, the prime decomposition of \(I_{M_2^{4}}\) 
consists of four binomials, two of which are \(I_{32 \perp 4 \mid 1}\) 
and \(I_{34 \perp 2 \mid 1}.\) 
However, \(I_{32 \perp 4 \mid 1}\) was the only binomial 
in the decomposition of \(I_{M_2^{7}},\) 
even though that ideal itself is binomial.

\item We also examined the following models, generated by the elementary 
CI statements of Class~I, which yield the three-sided CI relation 
described in Equation~\ref{eq:7}:
\[
M_3^{3} = \{ 1 \perp 2 \mid \emptyset,\; 2 \perp 4 \mid 1,\; 
2 \perp 4 \mid \emptyset,\; 1 \perp 2 \mid 4 \},
\]
\[
M_3^{4} = \{ 1 \perp 2 \mid \emptyset,\; 1 \perp 4 \mid 2,\; 
1 \perp 4 \mid \emptyset,\; 1 \perp 2 \mid 4 \}.
\]
The ideals \(I_{14 \perp 2 \mid \emptyset}\) and \(I_{24 \perp 1 \mid \emptyset}\) 
appear in the primary decompositions of \(I_{M_3^{3}}\) and \(I_{M_3^{4}},\) respectively. 
These same binomial components also occur in the decompositions of every submodel 
of \(M_3^{3}\) and \(M_3^{4}\) that is not generated solely by marginal CI statements.
\end{enumerate}

We aim to extend Proposition~\ref{prop:2} and prove the following conjecture for any given \(n.\)

\begin{conjecture}
Let \(M_d^{\alpha}\) denote the models induced by the faces of the elementary imset cone over \(n.\)
For every inclusion \(M_d^{\alpha} \subset M_{d+1}^{\beta},\) 
where \(M_{d+1}^{\beta}\) is generated by the elementary CI statements 
that form a non-elementary CI relation, 
the primary decomposition of \(I_{M_{d+1}^{\beta}}\) 
and of every \(I_{M_d^{\alpha}}\) that either contains a saturated CI statement 
or is not generated solely by marginal ones, 
is equal to \(I_E.\)
\end{conjecture}

 \section{Conclusion}

In conclusion, this study establishes a clear correspondence between the algebraic, 
combinatorial, and geometric structures of conditional independence (CI) statements. 
By analyzing the primary decomposition of CI ideals and their relationships to imsetal models, 
we demonstrated that certain binomial components, particularly those equal to \(I_E\), serve 
as fundamental building blocks for understanding non-elementary CI relations. 
The computational evidence for \(n = 3\) and \(n = 4\) supports the conjecture that 
the same structural behavior extends to higher dimensions.

Our results highlight that imsets and their associated cones provide an effective framework 
for identifying minimal primes and understanding the hierarchy of CI models across dimensions. 
Moreover, the observed correspondence between the faces of the elementary imset cone 
and the algebraic properties of CI ideals reinforces the geometric interpretation 
of independence relations.

Future work will aim to generalize Proposition~\ref{prop:2} and formally prove the conjecture 
for arbitrary \(n.\) 
We also plan to investigate computational techniques that exploit the combinatorial structure 
of imsets to improve efficiency in primary decomposition and model classification, 
extending these methods to mixed discrete systems and higher-dimensional settings.

\section*{Acknowledgments}
We express our gratitude to Dr.~Bernd Sturmfels for the opportunity to join his class 
based on his book \textit{Gröbner Bases and Convex Polytopes}. 
His insights and guidance have been invaluable to this research, 
as has his pioneering work in nonlinear algebra. 
Thank you, Dr.~Sturmfels.
\bibliographystyle{unsrt}
\bibliography{references} 

\newpage

\section*{ Appendix : Examples of CI Relations for n=4.} \label{apx:1}

$$ \textbf{Class I Quadratic CI Relation: Each connected to Type III statement in}\, \mathfrak{E_4}$$
\begin{scriptsize}
$$[	1 \indep 2| \emptyset + 2 \indep 3|1]\,\,\, = \,\,\, [2 \indep 3| \emptyset + 1 \indep 2|3] \,\,\, =\,\,\, [13 \indep 2 |\emptyset] \hspace{2cm}[1 \indep 3| \emptyset + 1 \indep 2|3] \,\,\, = \,\,\, [1 \indep 2| \emptyset + 1 \indep 3|2] \,\,\, =\,\,\, [23 \indep 1 |\emptyset] $$
$$	[1 \indep 3| \emptyset + 2 \indep 3|1] \,\,\, = \,\,\, [2 \indep 3| \emptyset + 1 \indep 3|2] \,\,\, =\,\,\, [12 \indep 3 |\emptyset] \hspace{2cm} [1 \indep 2| \emptyset + 2 \indep 4|1] \,\,\, = \,\,\, [2 \indep 4| \emptyset + 1 \indep 2|4] \,\,\, =\,\,\, [14 \indep 2 |\emptyset] $$
$$	[1 \indep 2| \emptyset + 1 \indep 4|2] \,\,\, = \,\,\, [1 \indep 4| \emptyset + 1 \indep 2|4 ] \,\,\, =\,\,\, [24 \indep 1 |\emptyset] \hspace{2cm} [1 \indep 4| \emptyset + 2 \indep 4|1] \,\,\, = \,\,\,[ 2 \indep 4| \emptyset + 1 \indep 4|2] \,\,\, =\,\,\, [12 \indep 4 |\emptyset]$$
$$	[1 \indep 3| \emptyset + 1 \indep 4|3] \,\,\, = \,\,\,[ 1 \indep 4| \emptyset + 1 \indep 3|4] \,\,\, =\,\,\, [34 \indep 1 |\emptyset] \hspace{2cm} [3 \indep 4| \emptyset + 1 \indep 3|4] \,\,\, = \,\,\, [1 \indep 3| \emptyset + 3 \indep 4|1] \,\,\, =\,\,\, [14 \indep 3 |\emptyset] $$
$$	[3 \indep 4| \emptyset + 1 \indep 4|3] \,\,\, = \,\,\, [1 \indep 4| \emptyset + 3 \indep 4|1] \,\,\, =\,\,\, [13 \indep 4 |\emptyset] \hspace{2cm} [2 \indep 3| \emptyset + 3 \indep 4|2] \,\,\, = \,\,\, [3 \indep 4| \emptyset + 2 \indep 3|4] \,\,\, =\,\,\, [24 \indep 3 |\emptyset] $$
$$	[2 \indep 4| \emptyset + 2 \indep 3|4] \,\,\, = \,\,\, [2 \indep 3| \emptyset + 2 \indep 4|3] \,\,\, =\,\,\, [34 \indep 2 |\emptyset]\hspace{2cm} [3 \indep 4| \emptyset + 2 \indep 4|3] \,\,\, = \,\,\, [2 \indep 4| \emptyset + 3 \indep 4|2]\,\,\, =\,\,\, [23 \indep 4 |\emptyset] $$
\end{scriptsize}
$$ \textbf{Class II Quadratic CI Relation: Each connected to Type IV statement in}, \mathfrak{E_4}$$
\begin{scriptsize}
$$[	3 \indep 4|1 + 2 \indep 3|14 ]\,\,\, = \,\,\,[ 2 \indep 3|1 + 3 \indep 4|12] \,\,\, =\,\,\, [24 \indep 3 |1] \hspace{1.7cm} [2 \indep 4|1 + 2 \indep 3|14] \,\,\, = \,\,\,[ 2 \indep 3|1 + 2 \indep 4|13 ] \,\,\, =\,\,\, [34 \indep 2 |1]$$
$$[	2 \indep 4|1 + 3 \indep 4|12] \,\,\, = \,\,\, [3 \indep 4|1 + 2 \indep 4|13] \,\,\, = \,\,\, [23 \indep 4 |1] \hspace{1.7cm} [1 \indep 3|2 + 3 \indep 4|12] \,\,\, = \,\,\, [3 \indep 4|2 + 1 \indep 3|24] \,\,\, = \,\,\, [14 \indep 3 |2] $$
$$	[1 \indep 3|2 + 1 \indep 4|23] \,\,\, = \,\,\,[ 1 \indep 4|2 + 1 \indep 3|24] \,\,\, = \,\,\, [34 \indep 1 |2] \hspace{1.7cm} [3 \indep 4|2 + 1 \indep 4|23] \,\,\, = \,\,\, [1 \indep 4|2 + 3 \indep 4|12] \,\,\, = \,\,\, [13 \indep 4 |2]$$
$$	[1 \indep 2|3 + 1 \indep 4|23]\,\,\, = \,\,\, [1 \indep 4|3 + 1 \indep 2|34] \,\,\, = \,\,\, [24 \indep 1 |3] \hspace{1.7cm} [1 \indep 4|3 + 2 \indep 4|13] \,\,\, = \,\,\, [2 \indep 4|3 + 1 \indep 4|23] \,\,\, = \,\,\, [12 \indep 4 |3]$$
$$	[1 \indep 2|3 + 2 \indep 4|13] \,\,\, = \,\,\, [2 \indep 4|3 + 1 \indep 2|34] \,\,\, = \,\,\, [14 \indep 3 |2] \hspace{1.7cm} [1 \indep 3|4 + 2 \indep 3|14] \,\,\, = \,\,\,[ 2 \indep 3|4 + 1 \indep 3|24 ] \,\,\, = \,\,\, [12 \indep 3 |4]$$
$$	[1 \indep 2|4 + 1 \indep 3|24 ] \,\,\, = \,\,\, [ 1 \indep 3|4 + 1 \indep 2|34 ] \,\,\, = \,\,\, [23 \indep 1 |4] \hspace{1.7 cm} [ 1 \indep 2|4 + 2 \indep 3|14 ]\,\,\, = \,\,\, [ 2 \indep 3|4 + 1 \indep 2|34] \,\,\, = \,\,\, [13 \indep 2 |4]$$
\end{scriptsize}

$$ \textbf{Cubic CI Relations: Each is connected to the Type I statement in}\, \mathfrak{E_4}$$

\begin{scriptsize}
$$	[2 \indep 4|13 + 1 \indep 4|3 + 4 \indep 3|\emptyset] \,\,\, = \,\,\, [ 3 \indep 4|12 + 1 \indep 4|2 +2 \indep 4|\emptyset] \,\,\, = \,\,\, [123 \indep 4 |\emptyset] $$
$$ [2 \indep 3|14 + 1 \indep 3|4 + 4 \indep 3|\emptyset] \,\,\, = \,\,\, [ 4 \indep 3|12 + 1 \indep 3|2 +2 \indep 3|\emptyset] \,\,\, = \,\,\, [124 \indep 3 |\emptyset]$$
$$	[1 \indep 2|34 + 1 \indep 3|4 + 1 \indep 4|\emptyset] \,\,\, = \,\,\, [ 2 \indep 3|14 + 1 \indep 3|4 +3 \indep 4|\emptyset] \,\,\, = \,\,\, [123 \indep 4 |\emptyset] $$
$$ [2 \indep 3|14 + 1 \indep 2|4 + 4 \indep 2|\emptyset] \,\,\, = \,\,\, [ 2\indep 3|14 + 1 \indep 2|4 +4 \indep 2|\emptyset] \,\,\, = \,\,\, [134 \indep 4 |\emptyset]$$
\end{scriptsize}

$$ \textbf{Quartic CI Relations: Each is connected to the Type II statement in}\, \mathfrak{E_4}$$

\begin{scriptsize}
$$	[1 \indep 3|24 + 1 \indep 4|2 + 2 \indep 4|3 + 2 \indep 3|\emptyset] \,\,\, = \,\,\, [ 1 \indep 4|23 + 1 \indep 3|2 +2 \indep 3|4 +2 \indep 4|\emptyset] \,\,\, = \,\,\, [12 \indep 34 |\emptyset] $$
$$	[1 \indep 3|24 + 1 \indep 4|2 + 2 \indep 4|3 + 2 \indep 3|\emptyset] \,\,\, = \,\,\, [ 1 \indep 4|23 + 1 \indep 3|2 +2 \indep 3|4 +2 \indep 4|\emptyset] \,\,\, = \,\,\, [12 \indep 34 |\emptyset] $$
$$	[1 \indep 2|34 + 1 \indep 3|4 + 3 \indep 4|2 + 2 \indep 4|\emptyset] \,\,\, = \,\,\, [ 2 \indep 4|13 + 1 \indep 3|2 +3 \indep 4|1 +1 \indep 2|\emptyset] \,\,\, = \,\,\, [14 \indep 23 |\emptyset] $$
$$	[2 \indep 3|14 + 1 \indep 2|4 + 1\indep 4|3 + 3 \indep 4|\emptyset] \,\,\, = \,\,\, [ 3 \indep 4|12 + 1 \indep 2|3 +1 \indep 4|2 +3 \indep 2|\emptyset] \,\,\, = \,\,\, [14 \indep 23 |\emptyset] $$
\end{scriptsize}

\end{document}